\documentclass[12pt]{amsart}
\usepackage{fullpage,amssymb,amsfonts,amsmath,amstext,amsthm,amscd,graphicx}
\usepackage[T1]{fontenc}

\newtheorem{theorem}{Theorem}[section]

\newtheorem{lemma}[theorem]{Lemma}

\newtheorem{corollary}[theorem]{Corollary}

\theoremstyle{definition}
\newtheorem*{remarks}{Remarks}

\newtheorem*{definition}{Definition}

\newcommand{\B}{{\bf b}}

\newcommand{\R}{\mathbb{R}}
\newcommand{\C}{\mathbb{C}}
\newcommand{\N}{\mathbb{N}}
\renewcommand{\B}{\mathbb{B}}

 \newcommand{\av}{\arrowvert}
 
 \renewcommand\mod{\operatorname{mod}}
 \newcommand\capac{\operatorname{cap}}
 \newcommand\inte{\operatorname{int}}
 \renewcommand\dim{\operatorname{dim}}

\numberwithin{equation}{section}

\begin{document}

\title[Julia sets of uniformly quasiregular mappings are uniformly perfect]
  {Julia sets of uniformly quasiregular mappings are uniformly perfect}

\author{Alastair N. Fletcher}
\address{Mathematics Institute\\
University of Warwick\\
Coventry\\
CV4 7AL}
\email{alastair.fletcher@warwick.ac.uk}

\author{Daniel A. Nicks}
\address{Dept of Mathematics and Statistics\\
The Open University\\
Milton Keynes\\
MK7~6AA}
\email{d.nicks@open.ac.uk}




\begin{abstract}
It is well-known that the Julia set $J(f)$ of a rational map $f:\overline{\C} \to \overline{\C}$ is uniformly perfect; that is,
every ring domain which separates $J(f)$ has bounded modulus, with the bound depending only on $f$.
In this article we prove that an analogous result is true in higher dimensions; namely, that the Julia
set $J(f)$ of a uniformly quasiregular mapping $f:\overline{\R^{n}} \to \overline{\R^{n}}$ is uniformly perfect.
In particular, this implies that the Julia set of a uniformly quasiregular mapping has positive Hausdorff dimension. 
\end{abstract}

\maketitle

\section{Introduction}

\subsection{Historical background}

The usage of the term uniformly perfect was first introduced by Pommerenke \cite{P}, but the idea originates in Beardon-Pommerenke \cite{BP} and Tukia-V\"{a}is\"{a}l\"{a} \cite{TV}, the latter under the guise
of homogeneously dense sets. 

The study of uniformly perfect sets has connections to hyperbolic geometry of domains in $\C$, limit sets of Fuchsian groups and even
holomorphic quadratic differentials and Teichm\"{u}ller theory. We refer to
\cite{S2} and the references contained therein for a more complete overview.

In relation to complex dynamics, it was proved independently by Eremenko \cite{E}, Hinkkanen \cite{H} and Man\'{e} and da Rocha \cite{MR} that the Julia set $J(f)$
of a rational map of the $2$-sphere $f:S^2 \rightarrow S^2$ is uniformly perfect, where we identify $S^{n}$ with $\R^{n} \cup \{ \infty \}$. 
An example of Baker \cite{B} (see \cite{H} for further details) shows that $J(f)$ need not be uniformly perfect
when $f$ is a transcendental entire function.

Quasiregular mappings of $\mathbb{R}^n$ share many 
properties with holomorphic functions of the plane, which gives rise to
the possibility of a rich theory of iteration of quasiregular mappings
in analogue to the well-studied field of complex dynamics.
For an overview of the current state of the theory of quasiregular dynamics, see the survey article of Bergweiler \cite{Berg}.
In \cite{JV}, J\"{a}rvi and Vuorinen investigated uniformly
perfect sets in $S^{n}$, in connection with quasiregular mappings. 

Uniformly quasiregular mappings were introduced by Iwaniec and Martin in \cite{IM} and are the subject of a number of papers. We restrict ourselves to mentioning Hinkkanen, Martin and Mayer's paper \cite{HMM}, where the interested reader can find further references. The uniformity condition on these mappings allows Julia and Fatou sets to be defined in direct analogue
with complex dynamics. Siebert proved in her thesis \cite{S}
that the Julia set of a uniformly quasiregular mapping $f:S^n \rightarrow S^n$ is perfect.

\subsection{Statement of results}

Before stating the main theorem, we make more precise the notion
of a uniformly perfect set. A ring domain $R \subset S^n$
is a domain whose complement $S^n \setminus R$ has
precisely two connected components. A ring domain $R$ is said to separate
a set $X$ if $R \cap X = \emptyset$ and both connected components of the complement of $R$ meet $X$. We defer the definition of the modulus
$\mod R$ of a ring domain $R$ until the next section, but remark here that it is
a measure of the thickness of a ring domain.

\begin{definition}
A closed set $E \subset S^n$ containing at least two points is called $\alpha$-uniformly perfect if there is no ring domain $R \subset S^n$ separating $E$ such that $\mod R > \alpha$. Further,
$E$ is called uniformly perfect if it is $\alpha$-uniformly perfect for some $\alpha >0$.
\end{definition}

The main theorem to be proved in this paper is as follows.

\begin{theorem}
\label{mainthm}
Let $f:S^n \rightarrow S^n$ be a uniformly quasiregular mapping that is not injective. Then the Julia set $J(f)$ is uniformly perfect. 
\end{theorem}

The central idea of the proof follows the proof that Julia sets of rational functions on $S^2$ are uniformly perfect, as given by \cite{CG}. As noted in the first section, there are several proofs of this result. However, the proof in \cite{CG} 
is more elementary than those given in \cite{H,MR}, and is also more readily extended to uniformly quasiregular mappings than \cite{E,S1}.
As a corollary to Theorem \ref{mainthm}, we have the following result.

\begin{theorem}
\label{thm2}
Let $f:S^n \rightarrow S^n$ be a uniformly quasiregular mapping that is not injective. Then the Julia set $J(f)$ has positive Hausdorff dimension.
\end{theorem}

This result follows from Theorem \ref{mainthm} and characterisations of
uniformly perfect sets given in \cite{JV}. We prove Theorem \ref{thm2} and obtain further applications
of \cite{JV} in the final section of this paper.

\begin{remarks}
 \mbox{} 
\begin{enumerate}
\item For examples of mappings to which Theorem \ref{mainthm} applies, see the Latt\`{e}s type maps considered
by Mayer in \cite{M1,M2}. In particular, note that these include higher dimensional analogues of power mappings and Chebyshev polynomials.
\item In Theorem \ref{mainthm}, we cannot relax the domain of $f$ to $\R^{n}$ and allow $f$ to have an essential singularity at infinity
due to Baker's example \cite{B}, recalling that holomorphic functions in the plane are uniformly $1$-quasiregular. Further, any transcendental entire function with multiply-connected Fatou components is a counterexample by a result of Zheng \cite{Zheng}. See also \cite{BZ} for results in this direction.
However, as far as the authors are aware, all known examples of uniformly quasiregular mappings of $\R^{n}$, for $n \geq 3$, extend to mappings of $S^n$.
\item Quasiregular iteration can be considered even when the mappings are not uniformly quasiregular. 
A key object of interest in this setting is the escaping set 
\[ I(f) = \{x \in \R^{n}:f^{n}(x) \to \infty \},\]
see \cite{BFLM,FN}.
In this case, the boundary of the escaping set $\partial I(f)$ has been posited as an analogue for the Julia set. 
A natural question is to ask whether there are conditions under which $\partial I(f)$ must or must not be uniformly perfect, 
in analogy with the work of Bergweiler and Zheng \cite{BZ,Zheng} for transcendental entire functions.
\end{enumerate}
\end{remarks}

The outline of the paper is as follows.
In Section~2 we introduce relevant definitions and notation, 
and state some intermediate lemmas that will be needed.
Section~3 contains the proof of Theorem~\ref{mainthm}, and in Section~4 we present some consequences of 
Theorem~\ref{mainthm} and results of J\"{a}rvi and Vuorinen \cite{JV}.

\section{Preliminaries}

\subsection{Quasiregular mappings}

A continuous mapping $f:G \rightarrow \R^n$ from a domain $G \subset \R^n$ is called quasiregular if $f$ belongs to the Sobolev space $W^{1}_{n,\mathrm{loc}}(G)$ and there exists $K \in [1, \infty)$ such that 
\begin{equation}
\label{eq2.1}
\av f'(x) \av ^{n} \leq K J_{f}(x)
\end{equation}
almost everywhere in $G$. Here $J_{f}(x)$ denotes the Jacobian determinant of $f$ at $x \in G$. The smallest constant $K \geq 1$ for which (\ref{eq2.1}) holds is called the outer dilatation $K_{O}(f)$. 
If $f$ is quasiregular, then we also have
\begin{equation}
\label{eq2.2}
J_{f}(x) \leq K' \inf _{\av h \av =1} \av f'(x) h \av ^{n}
\end{equation}
almost everywhere in $G$ for some $K' \in[1, \infty)$. The smallest constant $K' \geq 1$ for which (\ref{eq2.2}) holds is called the inner dilatation $K_{I}(f)$. The maximal dilatation $K(f)$ is the larger of $K_{O}(f)$ and $K_{I}(f)$, and a mapping $f$ is called $K$-quasiregular if $K(f)\le K$.
In dimension $n=2$, we have $K_{O}(f) = K_{I}(f)$.
If $G \subset S^n$ is a domain and $f : G \to S^n$ is
continuous, then $f$ is called $K$-quasiregular if it can locally be written as the composition
of a $K$-quasiregular mapping on $\R^n$ and sense-preserving M\"{o}bius transformations on $S^n$.
We note that some authors call such mappings $K$-quasimeromorphic. See Rickman's monograph \cite{R} for further details on 
the theory of quasiregular mappings.

The following analogue of Picard's theorem is central in the value distribution of quasiregular mappings.

\begin{theorem}[Rickman's theorem, {\cite[Theorem IV.2.1]{R}}]
\label{rickman}
For every $n \geq 2$ and $K \geq 1$, there exists a positive integer $q=q(n,K)$ which depends only on $n$ and $K$, such that the following holds. Every $K$-quasiregular mapping $f:\R^{n} \rightarrow S^{n} \setminus \{ a_{1},\ldots,a_{m} \}$ is constant whenever
$m \geq q$ and $a_{1},\ldots,a_{m}$ are distinct points in $S^{n}$.
\end{theorem}

Rickman's theorem leads to a quasiregular version of Montel's theorem. Recall that a family $\mathcal{F}$ of $K$-quasiregular mappings is called a normal family if every sequence in $\mathcal{F}$ has a subsequence which converges uniformly on compact subsets to a $K$-quasiregular mapping.

\begin{theorem}[Montel's theorem, \cite{M}]
\label{montel}
Let $\mathcal{F}$ be a family of $K$-quasiregular mappings in a domain $G \subset S^{n}$ and let $q=q(n,K)$ be Rickman's constant from Theorem~\ref{rickman}. If there exist distinct points $a_{1},\ldots,a_{q} \in S^{n}$ such that $f(G) \cap \{ a_{1},\ldots,a_{q} \} = \emptyset$ for all $f \in \mathcal{F}$, then $\mathcal{F}$ is a normal family.
\end{theorem}

\subsection{Iteration of quasiregular mappings}

The composition of two quasiregular mappings is again quasiregular.
For $k\in\mathbb{N}$, we write $f^{k}$ for the $k$-fold composition of a function $f$. 
A quasiregular mapping $f$ is called uniformly $K$-quasiregular if all the iterates of $f$ are $K$-quasiregular.

The Fatou set $F(f)$ of a uniformly quasiregular mapping $f:S^{n} \rightarrow S^{n}$ is defined as follows. A point $x\in S^{n}$ belongs to $F(f)$
if there exists a neighbourhood $U$ of $x$ such that the family
$\{ f^{k} \av _{U} : k \in \N \}$ is normal. 
The Julia set $J(f)$ is defined to be the complement of $F(f)$ in $S^n$, and so $S^n$ is partitioned into the open Fatou set and the closed Julia set. Both of these sets are completely invariant under $f$.

The exceptional set $\mathcal{E}_{f}$ is defined to be the largest discrete completely invariant set such that any open set $U$ which meets $J(f)$
satisfies
\[
S^n \setminus \mathcal{E}_{f} \subset \bigcup _{k =0}^{\infty} f^{k}(U). 
\]
By Montel's theorem, the exceptional set cannot contain more than $q(n,K)$ points and it is contained in $F(f)$; we refer to \cite{HMM} for further details on the exceptional set.

The following lemma reveals an expanding property on the Julia set.

\begin{lemma}[{\cite[Proposition 3.2]{HMM}}]
\label{lemma5}
Let $f:S^n \to S^n$ be uniformly quasiregular and let $U$ be an open set that meets $J(f)$. 
Then there exists $N \in \N$ such that $\overline{U} \subset f^{N}(U)$,
and such that $U _{k} = f^{kN}(U)$ is an increasing sequence exhausting $S^{n} \setminus \mathcal{E}_{f}$.
\end{lemma}

\subsection{Modulus and capacity}

See \cite{V} for more details on the notions introduced in this subsection.

Write $\chi$ and $d$ for the chordal and Euclidean distances on $S^n$
and $\R^{n}$ respectively.
The chordal distance is normalized so that $\chi(x,y)\leq 1$ for all $x,y \in S^n$ with equality if and only if $x$ and $y$ are antipodal points. 
For sets $E$ and $F$ in $S^n$, we write $\chi(E)$ for the chordal diameter of $E$ and $\chi(E,F)$ for the chordal distance between $E$ and $F$.
If $E$ consists of one point $x \in S^n$, we write $\chi(x,F)$.
Similarly, for sets $E$ and $F$ in $\R^{n}$, we write $d(E)$ for the Euclidean diameter of $E$ and $d(E,F)$ for the Euclidean distance between $E$ and $F$. Denote by $B_{d}(x,r)$ and $B_{\chi}(x,r)$ respectively the balls
centred at $x$ of Euclidean and chordal radius $r$. We also write $\mathbb{B}^n=B_d(0,1)$.
Denote by $A_{d}(x,r,s)$ the Euclidean annulus $\{y\in \R^{n} : r<d(y,x) <s\}$. The chordal annulus
$A_{\chi}(x,r,s)$ is defined analogously.

A domain $R\subset S^n$ is called a ring domain if $S^n \setminus R$ has exactly two components. If the two components are $C_{0}$ and $C_{1}$, then we write $R = R(C_{0},C_{1})$.

Given two sets $E$ and $F$, we write $\Delta (E,F ; V)$ for the family of paths with one end-point in $E$, the other end-point in $F$, and which are contained in $V$. 
When $V$ is $S^n$, we abbreviate the notation to $\Delta(E,F)$.

The $n$-modulus $M(\Gamma)$ of a path family $\Gamma$ is defined by
\[
M(\Gamma) = \inf \int _{\R^n} \rho ^{n} \: dm,
\]
where $m$ denotes $n$-dimensional Lebesgue measure, and the infimum is taken
over all non-negative Borel measurable functions $\rho$ such that
\[
\int _{\gamma} \rho \: ds \geq 1
\]
for each locally rectifiable curve $\gamma \in \Gamma$. The $n$-modulus is a conformal invariant.

The conformal modulus of a ring domain $R(C_{0},C_{1})$ is defined by
\begin{equation}
\label{moddef}
\mod R(C_{0},C_{1}) = \left ( \frac{ M(\Delta(C_{0},C_{1}))}{\omega _{n-1}} \right ) ^{1/(1-n)},
\end{equation}
where $\omega_{n-1}$ is the $(n-1)$-dimensional surface area of the unit $(n-1)$-sphere. 
If $R_{0},R_{1}$ are two ring domains with $R_{0} \subset R_{1}$, then
\begin{equation}
\label{modeq1}
\mod R_{0} \leq \mod R_{1}.
\end{equation}
The capacity of a ring domain $R(C_{0},C_{1})$ is defined to be 
\begin{equation}
\label{capdef}
\capac R = M(\Delta(C_{0},C_{1})).
\end{equation}

The Teichm\"{u}ller ring $R_{T,n}(s)$ is a special example of a ring domain. It has complementary components $[-e_{1},0]$ and $[se_{1},\infty]$, where $e_{1} = (1,0,\ldots,0)$ and $s>0$. We write $\tau_{n}(s)$ for the $n$-modulus of the family of paths connecting the complementary components of $R_{T,n}(s)$ in $S^{n}$, that is
\[
\tau_{n}(s) = \capac R_{T,n}(s).
\]
It is shown in \cite[Lemma 5.53]{V} that, for $s>0$, the function $\tau_{n}(s)$ is decreasing in $s$ and satisfies 
\begin{equation}
\label{taulim}
\tau_n(s)>0 \quad \mbox{and} \quad \lim _{s \rightarrow \infty} \tau_{n}(s) = 0.
\end{equation}

The following results relate the Teichm\"{u}ller ring capacity, the capacity of ring domains and
the distance between their complementary components.

\begin{lemma}[{\cite[Lemma 2.9]{JV}}]
\label{lemma1a}
Let $R$ be a ring domain in $S^n$ with complementary components $E$ and $F$. Then
\[
\capac R \geq 2^{1-n} \tau_{n} \left ( \frac{2 \chi(E,F)}{ \min \{ \chi (E), \chi (F) \}} \right).
\]
\end{lemma}

\begin{lemma}[{\cite[Corollary 7.39]{V}}]
\label{lemma2}
Let $E$ and $F$ be disjoint continua in $\R^{n}$ which satisfy
$0<d(E) \leq d(F)$. Then
\[
M(\Delta (E,F)) \geq 2^{1-n} \tau _{n} \left ( \frac{ d(E,F) }{d(E)} \right ).
\]
\end{lemma}

The next lemma is mentioned as Remark 2.7(ii) of \cite{JV}, which refers to \cite[Corollary~7.37]{V}, but we provide a proof for the convenience of the reader.

\begin{lemma}
\label{lemma3}
A closed set $X$ is uniformly perfect if and only if the moduli of the chordal annuli separating $X$ are bounded from above.
\end{lemma}

\begin{proof}
It is clear from the definition that if $X$ is uniformly perfect, then the implication holds. Suppose that $X$ is not uniformly perfect. Then there exist ring domains
$R_{k} = R(E_{k},F_{k})$ separating $X$ such that $\mod R_{k} \rightarrow \infty$. Without loss of generality, we may assume that $\chi (E_{k}) \leq \chi (F_{k})$ for all $k \geq 1$.
Then by (\ref{moddef}), (\ref{capdef}), (\ref{taulim}) and Lemma \ref{lemma1a},
it follows that
\begin{equation}
\label{l1aeq2}
\frac{\chi (E_{k},F_{k})}{\chi(E_{k})} \rightarrow \infty.
\end{equation}
Choose $k_{0} \in \N$ such that 
\begin{equation}
\label{l1aeq1}
10\chi (E_{k}) \leq  \chi (E_{k},F_{k})
\end{equation}
for $k \geq k_0$, and assume henceforth that $k \geq k_{0}$.
Choose $x_{k} \in E_{k}\cap X$ and let  $u_{k}= \chi (E_{k})$ and  $w_{k} = \chi (x_{k},F_{k}) \geq \chi(E_{k},F_{k})$.
Then by (\ref{l1aeq1}), the chordal annulus 
$A_{k} = A_{\chi}(x_{k},u_{k},w_{k})$ is contained in $R_{k}$ and separates $X$. 
Note that (\ref{l1aeq1}) ensures that $A_k$ is a well-defined non-empty chordal annulus.
Further, since $w_k\in [0,1]$ and $u_{k} \to 0$ by (\ref{l1aeq2}), we may assume that, for $k\ge k_0$,
\begin{equation}
\label{l1aeq3}
\sqrt \frac{ 1- u_{k}^{2}}{1-w_{k}^{2}} \geq \frac{1}{2}.
\end{equation}

Using \cite[Corollary 5.18]{V}, (\ref{moddef}) and (\ref{l1aeq3}), we have that
\[
 \mod A_{k} = \log \left ( \frac{w_{k}}{u_{k}} \sqrt{ \frac{ 1-u_{k}^{2}}{1-w_{k}^{2}}} \right )  \geq \log \left ( \frac{w_k}{2u_k} \right ),  
\]
which together with (\ref{l1aeq2}) implies that $\mod A_{k} \rightarrow \infty$.
\end{proof}

The details of the proof of Lemma \ref{lemma3} yield the following corollary.

\begin{corollary}
\label{lemma3cor}
Let $X\subset \R^{n}$ be a compact set which is not uniformly perfect. 
Then there exists a sequence of chordal annuli $B_{k} = A_{\chi}(x_{k},u_{k},v_{k})$
separating $X$ with $x_{k} \in X$ and $\mod B_{k} \to \infty$ such that each annulus $B_k$ separates $x_k$ from $\infty$ and is contained in a bounded region of $\R^{n}$.
\end{corollary}

\begin{proof}
Following the notation and proof of Lemma \ref{lemma3}, we find $x_{k} \in X$
and chordal annuli $A_{k} = A_{\chi}(x_{k},u_k,w_k)$ which separate $X$ and satisfy $w_k/u_k \to \infty$.

Let $\eta = \chi (\infty,X)$, which is positive since $X$ is compact.
Let $v_{k} = \min \{ \eta /2 , w_{k} \}$,
then $B_{k} = A_{\chi}(x_{k},u_{k},v_{k}) \subseteq A_{k}$. It is possible that
$u_{k} \geq v_{k}$, but only for finitely many $k \in \N$, since $u_{k} \to 0$ and $w_k>u_k$.
We remove these finitely many terms from the sequence and re-label. 

To prove that $\mod B_{k} \to \infty$, we observe that
\[
\frac{v_k}{u_k}= \min\left\{\frac{\eta}{2u_k}, \frac{w_k}{u_k}\right\} \to \infty
\]
and then follow the last part of the proof of Lemma \ref{lemma3}.
The other claims of the corollary follow immediately.
\end{proof}

\section{Proof of Theorem \ref{mainthm}}

Let $f:S^n \to S^n$ be a uniformly quasiregular mapping. If $J(f) = S^n$ then we are done, noting that this case can occur, for example with Latt\`{e}s type mappings \cite{M1}.
Otherwise, there exists $x \in F(f)$. Let $g:S^n \to S^n$ be a M\"{o}bius transformation which sends $x$ to~$\infty$. Then $\tilde{f}=g \circ f \circ g^{-1}$ is a uniformly quasiregular mapping for which $\infty \in F(\tilde{f})$. Since $F(\tilde{f})$ is open, this implies that  $J(\tilde{f})$ is contained in a bounded region of $\R^n$.
Further, $J(f) = g^{-1}( J(\tilde{f}) )$, and since $g$ is conformal, $J(f)$ is uniformly perfect if and only if $J(\tilde{f})$ is uniformly perfect.

We may therefore assume without loss of generality that $J(f)$ is a compact subset of $\R^n$.
Suppose for a contradiction that $J(f)$ is not uniformly perfect. 
Then by 
Lemma \ref{lemma3} and Corollary \ref{lemma3cor}, there exist $x_{k} \in J(f)$ and a sequence of 
chordal annuli 
\[ B_{k}=A_{\chi}(x_{k},u_{k},v_{k}) \subset \R^{n} \setminus J(f)\] 
which separate $J(f)$ 
such that $\mod B_{k} \to \infty$ and $B_{k} \subset B_{d}(0,T)$ for some $T>0$.

By \cite[p.8]{V}, the chordal annuli $B_{k}$ take the following form:
there exist points $y_{k},z_{k} \in \R^{n}$ and $s_{k},t_{k}>0$
such that 
\[
B_{k} = B_{d}(z_{k},s_{k}) \setminus \overline{B_{d}(y_{k},t_{k})}.
\]

\begin{lemma}
We have $s_{k}/t_{k} \rightarrow \infty$. 
\end{lemma}

\begin{proof}
Let $\Omega_{k}=A_{d}(y_{k},t_{k},s_{k}+d(z_{k},y_{k}))$ so that $B_{k} \subset \Omega_{k}$.
By (\ref{modeq1}), $\mod B_{k} \leq \mod \Omega_{k}$. By (\ref{moddef}) and \cite[(5.14), p.53]{V}, 
\[
\mod \Omega _{k} = \log \left ( \frac{ s_{k} + d(z_{k},y_{k})}{t_{k}} \right).
\]
Then since $d(z_{k},y_{k}) \leq s_{k}$, it follows that $\mod \Omega _{k} \leq \log 2 s_{k}/t_{k}$.
Since $\mod B_{k} \to \infty$, the lemma follows.
\end{proof}

Since $s_{k}\leq T$, the lemma implies that $t_k\to0$.
Let $A_{k}\subset B_{k}$ be the Euclidean annulus
\[
A_{k} = A_{d}(y_{k},t_{k},r_{k}), 
\]
where $r_{k}$ is the Euclidean distance from $y_{k}$ to the outer boundary component of $B_{k}$.
Write $E_k=\overline{B_d(y_k,t_k)}$ for the common bounded component of the complements of $B_k$ and $A_k$. 
By Corollary \ref{lemma3cor}, the annulus $A_k$ separates $x_k$ from $\infty$, and hence $E_k=\overline{B_{\chi}(x_k,u_k)}$.

Now, suppose that $\delta >0$ is small. Since $d(E_k)=2t_k\to0$, we may assume without loss of generality that $d(E_k)\le\delta$ for all $k$. Note also that $\inte (E_{k})$ meets $J(f)$ because $x_{k} \in \inte (E_{k})$.
Using Lemma \ref{lemma5} applied to $\inte (E_{k})$, there exists a positive integer $j_{k}$ for which
\begin{equation}
\label{eq1}
d (f^{j_{k}-1}(E_{k})) \leq \delta, \quad \mbox{but} \quad d(f^{j_{k}}(E_{k})) > \delta.
\end{equation}

Let $P$ be the set of poles of $f$. Since $\infty \in F(f)$ and the Julia set is closed, we have that $\eta = d(J(f),P) >0$.
Now let 
\[ H = \{ y \in \R^{n} : d(y, J(f))\leq \eta / 3\}\] 
and  
\[ G = \{ y \in \R^{n} : d(y, H)< \eta / 3\},\]
so we have $J(f) \subset H \subset G$ and $\overline{G} \cap P = \emptyset$.
Since $E_{k}$ meets $J(f)$, we have that $f^{j_{k}-1}(E_{k}) \subset H$ by using (\ref{eq1}) and the fact that $\delta$ is small.
Then by the H\"{o}lder continuity of $f$, see for example \cite[Theorem III.1.11]{R}, there exists $C>0$ depending only on $f$, $n$ and $K$ such that
\begin{equation}
\label{Cdelta}
d( f^{j_{k}} (E_{k})) \leq C\delta ^{\alpha},
\end{equation}
where $\alpha = K_{I}^{1/(1-n)}$. In fact, we may take $C = \lambda_{n} d(H,\partial G)^{-\alpha} d(f(G))$, where $\lambda_{n}$ depends only on $n$. To see that $C$ is finite, observe that $\overline{G} \cap P = \emptyset$
and $d(H,\partial G) = \eta / 3 >0$, which respectively imply that $d(f(G))$ and $d(H,\partial G)^{-\alpha}$ are finite.

\begin{lemma}\label{r/t}
Each annulus $A_{k}$ is contained in the Fatou set $F(f)$. Moreover,  $r_k/t_k\to\infty$ as $k\to\infty$.
\end{lemma}

\begin{proof}
We note that $A_k\subset B_k\subset F(f)$ because $B_k$ separates $J(f)$.
Denote by $F_k$ the unbounded component of the complement of $B_k$.
Since $\mod B_{k} \rightarrow \infty$, it follows from (\ref{moddef}), (\ref{taulim}) and Lemma \ref{lemma2} that
\[
\frac{d(E_{k},\partial F_{k})}{d(E_{k})} \rightarrow \infty.
\]
The lemma is now proved by observing that $d(E_{k})=2t_k$ and $d(E_{k},\partial F_{k})=r_k-t_k$.
\end{proof}

The idea now is to normalize everything so that we can consider mappings on the unit ball.
To this end, let $\psi_k(x) = r_kx+y_k$. Then $\psi_k$ is a linear map from $\mathbb{B}^n$ onto $A_{k} \cup E_{k}=B_{d}(y_{k},r_{k})$. 
Let $V_k= \psi_k^{-1}(E_k)=\overline{B_d(0,t_{k}/r_{k})}$.

Define
\[
g_{k} = f^{j_{k}} \circ \psi _{k} : \B^{n} \rightarrow S^n.
\]
Then by (\ref{eq1}) and (\ref{Cdelta}), we have that
\begin{equation}
\label{eq2}
\delta <d(g_{k}(V_{k})) \le C\delta^{\alpha}.
\end{equation}

Let $q=q(n,K)$ be Rickman's constant from Theorem \ref{rickman}. Suppose that $\delta$ is small enough that we can choose $q+1$ points $a_{1},\ldots,a_{q+1}$ 
in $J(f)$ which lie at mutual Euclidean distance greater than $C\delta^{\alpha}$
from each other. 
Then, using (\ref{eq2}) and the fact that  $g_{k}(\B^{n} \setminus V_{k}) \subset F(f)$ from Lemma \ref{r/t}, it follows that $g_{k}(\mathbb{B}^n)$ contains at most one of the points $a_{j}$. By passing to a subsequence and re-labelling, we may assume that
\begin{equation}
\label{g_k(B^n)}
g_k(\mathbb{B}^n)\cap\{a_1,\ldots,a_{q}\}=\emptyset.
\end{equation}

Since each $g_{k}$ is the composition of an iterate of a uniformly $K$-quasiregular mapping and a conformal mapping, it follows that $g_{k}$ is $K$-quasiregular. Theorem~\ref{montel} and (\ref{g_k(B^n)}) then imply that the $g_{k}$ form a normal family on $\B^{n}$. Therefore, since $d(V_{k}) =2t_k/r_k\rightarrow 0$ by Lemma~\ref{r/t}, and because the images $g_{k}(V_{k})$ lie in a uniformly bounded region of $\R^{n}$, we have that
$d (g_{k}(V_{k})) \rightarrow 0$. This contradicts (\ref{eq2}) and proves the theorem.

\section{Properties of the Julia set}

\subsection{Hausdorff dimension}

We briefly overview the definition of Hausdorff dimension.
For $\beta >0$, let $\Lambda ^{\beta}$ denote the \mbox{$\beta$-dimensional} Hausdorff
content; that is, 
$\Lambda^{\beta} (E) = \inf \left \{ \sum_{i=1}^{\infty} r_{i}^{\beta} \right \}$,
where the infimum is taken over all coverings of $E \subset \R^{n}$ with countably many Euclidean balls of radius $r_{i}$. Set
\[
H_{\delta}^{\beta}(E) = \inf  \left \{ \sum_{i=1}^{\infty} d(U_{i})^{\beta} \right \},
\]
where the infimum is taken over all countable coverings of $E$ by sets $U_{i}$
with $d(U_{i})< \delta$. Then the $\beta$-dimensional Hausdorff measure is
$H^{\beta}(E) = \lim_{\delta \rightarrow 0} H_{\delta}^{\beta}(E)$. The Hausdorff
dimension is defined to be
\[
\dim _{H}(E) = \inf \{\beta :  H^{\beta}(E)< \infty \}.
\]

\begin{proof}[Proof of Theorem \ref{thm2}.]
By \cite[Theorem 4.1]{JV} and Theorem \ref{mainthm}, there exist positive constants $\beta$ and $C_{1}$ depending on $f$ such that the $\beta$-dimensional Hausdorff content of $J(f)$ satisfies
\[
\Lambda^{\beta} ( \overline{B_{d}(x,r)}\cap J(f)) \geq C_{1} r^{\beta}
\]
for all $x \in J(f) \cap \R^{n}$ and $r \in (0,d(J(f)))$.
As an immediate corollary to this, we have that
the Hausdorff dimension of $J(f)$ is positive.
\end{proof}

\subsection{Further metric properties}

We introduce some further notation, following the presentation of \cite{JV}, to be able to state the theorem of this section.

Let $G$ be a domain in $\R^{n}$ with non-empty complement. Then the quasihyperbolic metric $k_{G}$ on $G$ is the metric with density
\[
w(x) = \frac{1}{d(x,\partial G)}.
\]

Let $E$ be a compact subset of a domain $G$ in $\R^{n}$. The pair $(G,E)$ is called a condenser and its capacity
is defined to be
\begin{equation}
\label{capdef2}
\capac (G,E) = M( \Delta ( E, \partial G ; G)).
\end{equation}
Given a closed set $E \subset \R^{n}$, a point $x \in \R^{n}$ and $r>0$, we set
\[
\capac (x,E,r) = \capac (B_d(x,2r),\overline{B_d(x,r)}\cap E).
\]

Now let $E \subset S^n, x \in S^n$ and $0<r<t<1$
and set
\[
m_{t}(E,r,x) = M(\Delta ( \partial B_{\chi}(x,t), E \cap \overline{ B_{\chi}(x,r)})).
\]
We say that $E$ satisfies a metric thickness condition if there exist $\delta >0$ and $r_{0} \in (0,1/4)$ such that
\[
m_{2r}(A_{x}(E),r,0) \geq \delta
\]
for all $x \in E$ and $r \in (0,r_{0})$, where $A_{x}$ is a chordal isometry
such that $A_{x}(x)=0$.

\begin{theorem}
Let $f$ satisfy the hypotheses of Theorem \ref{mainthm}, and suppose that $J(f)$ is $\alpha$-uniformly perfect. Then we have the following results.
\begin{enumerate}
\item Let $U$ be a completely invariant component of the Fatou set $F(f)$.
Then $f:(U,k_{U}) \rightarrow (U,k_{U})$ is uniformly continuous.
\item There exists a constant $C_{2}>0$ depending on $f$ such that
\[
\capac (x,J(f),r) \geq C_{2}
\]
for all $x \in J(f)$ and $r \in (0,d(J(f)))$.
\item $J(f)$ satisfies a metric thickness condition with $r_{0} = \chi(J(f))/4$
and $\delta = \delta(\alpha,n)>0$.
\end{enumerate}
\end{theorem}

\begin{proof}
If $U$ is a completely invariant component of $F(f)$, then $J(f) = \partial U(f)$. By Theorem \ref{mainthm} and \cite[Theorem 3.5]{JV}, we have (i).
Part (ii) follows from \cite[Theorem 4.1]{JV} and part (iii) from \cite[Theorem 4.4]{JV}.
\end{proof}

{\it Acknowledgements:}
The authors would like to thank Professor Jim Langley for initially suggesting investigating whether
the Julia sets of quasiregular mappings might be uniformly perfect, and also for his helpful comments.


\end{document}